\documentclass[12pt]{article}
\usepackage[utf8]{inputenc}
\usepackage{color}
\usepackage{amsthm}
\usepackage{mathrsfs}
\usepackage{stmaryrd}
\usepackage{microtype}
\usepackage[T1]{fontenc}
\usepackage{authblk}
\usepackage{stackengine}
\usepackage{a4,a4wide}
\usepackage{amscd,verbatim,fancyhdr}
\usepackage{graphicx}
\usepackage{turnstile}
\usepackage[reqno, namelimits, sumlimits]{amsmath}
\newcommand{\ee}{\epsilon}
\newcommand{\al}{\alpha}

\newcommand{\DD}{\Delta}
\newcommand{\lbd}{\lambda}

\newcommand{\R}{\mathbb{R}}

\newcommand{\N}{\mathbb{N}}

\DeclareMathOperator*{\dvg}{\mbox{div}}
\DeclareMathOperator*{\curl}{\mbox{curl}}
\DeclareMathOperator*{\essup}{\mbox{ess sup}}
\usepackage{amssymb,amsfonts}
\usepackage{bbold}
\usepackage{multirow}
\usepackage{blindtext}
\usepackage{hyperref}
\hypersetup{
    colorlinks=true,
    linkcolor=green,
    citecolor=blue
}
\newtheorem{theorem}{Theorem}[section]

\newtheorem{corollary}{Corollary}[theorem]
\newtheorem{lemma}[theorem]{Lemma}
\theoremstyle{remark}
\newtheorem{remark}{Remark}
\theoremstyle{definition}
\newtheorem{definition}{Definition}[section]
\numberwithin{equation}{section}
\def\Xint#1{\mathchoice
{\XXint\displaystyle\textstyle{#1}}%
{\XXint\textstyle\scriptstyle{#1}}%
{\XXint\scriptstyle\scriptscriptstyle{#1}}%
{\XXint\scriptscriptstyle\scriptscriptstyle{#1}}%
\!\int}
\def\XXint#1#2#3{{\setbox0=\hbox{$#1{#2#3}{\int}$ }
\vcenter{\hbox{$#2#3$ }}\kern-.6\wd0}}
\def\avint{\Xint-}
\begin{document}

\title{Decay Estimate for some Toy-models related to the Navier-Stokes system}

\author{F.~Hounkpe\footnote{Email address: \texttt{hounkpe@maths.ox.ac.uk}}
}
\affil{Mathematical Institute, University of Oxford, Oxford, UK}

\maketitle

\begin{abstract}
We prove in this paper a decay estimate for scaling invariant local energy solutions for some toy-models related to the incompressible Navier-Stokes system. 
\end{abstract}
\section{Introduction}
We consider the following models
\begin{equation}\label{Toy-Mod1}
    \partial_t u - \DD u -\kappa \nabla \dvg u + u\cdot \nabla u + \frac{u}{2} \dvg u = 0\quad \mbox{in }\R^3 \times (0,+\infty)
\end{equation}
and 
\begin{equation}\label{Toy-Mod2}
    \partial_t u - \DD u -\kappa \nabla \dvg u + \dvg \left(u\otimes u + \frac{|u|}{2} I_3\right) = 0\quad \mbox{in }\R^3 \times (0,+\infty), 
\end{equation}
where $I_3$ denotes the identity matrix in $3$D and $\kappa\geq 0$. Notice that if $u$ is a solution of either one of the previous systems, then for all $\lbd>0$, the vector field $u^{\lbd}: (x,t)\mapsto \lbd u(\lbd x, \lbd^2 t)$ is also a solution to that system. An interesting question now is whether or not one can construct a scale invariant solution $u$, i.e. $u^{\lbd} = u$ for all $\lbd>0$, to system \eqref{Toy-Mod1} and \eqref{Toy-Mod2}. This question was tackled and solved in \cite{Jia14} for the Cauchy problem for the $3$D incompressible Navier-Stokes system 
\begin{equation}\label{NS}
\left.
    \begin{gathered}
        \partial_t v - \DD v + v \cdot \nabla v + \nabla p = 0,\\
        \dvg v = 0,\\
        v|_{t=0} = u_0,
    \end{gathered}
\right\}    
\end{equation}
where $\lbd u_0(\lbd x) = u_0$ for all $\lbd>0$. To achieve this, Jia and \v{S}ver\'ak looked for solutions of the form $v(x,t) = \frac{1}{\sqrt{t}}U(\frac{x}{\sqrt{t}})$, where $U$ satisfies the following system
\begin{equation}\label{E1.4}
    \left.
    \begin{gathered}
    -\Delta U + U\cdot \nabla U -\frac{x}{2}\cdot \nabla U - \frac{U}{2} + \nabla P = 0,\\
    \dvg U = 0.
    \end{gathered}
    \right\}\mbox{in }\R^3.
\end{equation}
And they require the following asymptotics on $U$
\begin{equation}\label{E1.5}
    |U(x)-u_0(x)| = o(|x|^{-1})\quad\mbox{as }|x|\to \infty.
\end{equation}
They solve the problem \eqref{E1.4}-\eqref{E1.5} thanks to Leray-Schauder degree theory applied in a suitable function space. This was make possible by some quite technical decay estimates (see Theorem 4.1 in \cite{Jia14}). Our goal in this note is to establish a similar decay estimate for system \eqref{Toy-Mod1} and \eqref{Toy-Mod2}. The particularity here being that the leading term in our systems is the Lam\'e operator for which the required regularity results we need seem not to appear in the literature. We will focus mostly on system \eqref{Toy-Mod1} since the computations for system \eqref{Toy-Mod2} are the same if not simpler because of the divergence structure of the non-linearity.\\
In section 2 and 3, we prove the intermediate results we need for the proof of our main theorem; those are $\ee-$regularity results for a generalised version of system \eqref{Toy-Mod1} and a so-called local near initial time regularity result. In section 4, we prove our main result.
\subsection{Preliminaries}
Before continuing our development, let us explain our notations
\[ z=(x,t), \quad z_0=(x_0,t_0),\quad B(x_0,R) = \{ |x-x_0|<R \};\]
\[ Q(z_0,R) = B(x_0,R)\times (t_0-R^2,t_0); \]
\[ B(r) = B(0,r),\quad Q(r) = Q(0,r),\quad B=B(1),\quad Q = Q(1); \]
For $\Omega$ an open subset of $\R^n$ and $-\infty\leq T_1<T_2 \leq +\infty$. Set $Q_{T_1,T_2}:= \Omega \times (T_1,T_2)$. We will be using
$L_{m,n}(Q_{T_1,T_2}):= L_n(T_1,T_2;L_m(\Omega))$, the Lebesgue space with the norm
\[
\|v\|_{m,n,Q_{T_1,T_2}} = \begin{cases}
\left( \int_{T_1}^{T_2}\|v(\cdot,t)\|^n_{L_m(\Omega)}dt\right)^{1/n},\quad & 1\leq n< \infty\\
\mbox{ess}\displaystyle\sup_{(T_1,T_2)}\|v(\cdot,t)\|_{L_m(\Omega)},\quad & n= \infty,
\end{cases}
\]
\[ L_m(Q_{T_1,T_2})=L_{m,m}(Q_{T_1,T_2}), \quad \|v\|_{m,m,Q_{T_1,T_2}} = \|v\|_{m,Q_{T_1,T_2}}; \]
$W^{1,0}_{m,n}(Q_{T_1,T_2})$, $W^{2,1}_{m,n}(Q_{T_1,T_2})$ are the Sobolev spaces with mixed norm,
\[ W^{1,0}_{m,n}(Q_{T_1,T_2}) = \left\{ v,\nabla v \in L_{m,n}(Q_{T_1,T_2}) \right\}, \]
\[ W^{2,1}_{m,n}(Q_{T_1,T_2}) = \left\{ v,\nabla v, \nabla^2 v, \partial_t v \in L_{m,n}(Q_{T_1,T_2}) \right\}, \]
\[ W^{1,0}_m(Q_{T_1,T_2}) = W^{1,0}_{m,m}(Q_{T_1,T_2}), \quad W^{2,1}_m(Q_{T_1,T_2}) = W^{2,1}_{m,m}(Q_{T_1,T_2}).\]
We use the following mean value notations:
\[
(f)_{z_0,R} = \frac{1}{|Q(R)|}\int_{Q(z_0,R)}g(z)dz,\quad (g)_{,R} = (g)_{0,R},
\]
where $|\Omega|$, in the above, stands for the $4$-dimensional Lebesgue measure of the domains $\Omega$.\\
We use $c$ or $C$ to denote an absolute constant and we write $C(A,B,\ldots)$ when the constant depends on the parameters $A,B,\ldots$\\
Finally, we denote by $e^{\nu \Delta t}$ the semigroup associated to the heat equation 
\begin{equation}\label{E1.6}
    \partial_t u - \nu \Delta u = 0.
\end{equation}
To be more precise, $u(x,t) = e^{\nu \Delta t}u_0 (x)$ denotes the unique solution of \eqref{E1.6} in $\mathbb{R}^n\times (0,\infty)$ with initial data $u_0$. Moreover, if $u_0$ belongs to a suitable function space, the following formula is available. 
\begin{equation*}
    e^{\nu \Delta t}u_0 (x) = \int_{\mathbb{R}^n}\Gamma_{\nu}(x-y,t)u_0(y) dy, 
\end{equation*}
where 
\[
\Gamma_{\nu}(x,t) = \frac{1}{(4\pi\nu t)^{\frac{n}{2}}}\exp\left( -\frac{|x|^2}{4\nu t} \right).
\]
In order to state our main result we need the following definition. Our setting is as follows:
we take $u_0 \in L_{2,loc}(\R^3)$ such that $\sup_{x_0 \in \R^3}\int_{B_1(x_0)}|u|^2 dx < \infty$.
\begin{definition}\label{Def1.1}
A vector field $u\in L_{2,loc}(\R^3 \times [0,\infty))$ is called a Leray solution or local energy solution to system \eqref{Toy-Mod1} with initial data $u_0$ (as above) if 
\begin{enumerate}
    \item For all $R>0$, we have
    \[ 
    \sup_{x_0 \in \R^3}\left( \essup_{0\leq t < R^2}\int_{B(x_0,R)}|u(x,t)|^2 dx + \int_0^{R^2}\int_{B(x_0,R)}|\nabla u(x,t)|^2 dx dt \right) < \infty;
    \]
    \item The vector field $u$ solves the Cauchy problem for \eqref{Toy-Mod1} in the following sense:
    \[
    \partial_t u - \DD u - \kappa \nabla \dvg u + u\cdot \nabla u + \frac{u}{2}\dvg u = 0\quad\mbox{in }\mathcal{D}'(\R^3 \times (0,\infty)),
    \]
    and 
    \[
    \|u(\cdot,t) - u_0\|_{L_2(K)} \xrightarrow[t\to 0^+]{} 0,
    \]
    for any compact set $K\subset \R^3$;
    \item For any smooth $\phi\geq 0$ with supp$\phi \subset\subset \R^3 \times (0,\infty)$, the following local energy inequality holds
    \begin{multline*}
        \int_0^{\infty}\int_{\R^3}\left(|\nabla u|^2 + \kappa(\dvg u)^2\right)\phi(x,t)dx dt \leq \int_0^{\infty}\int_{\R^3}\frac{|u|^2}{2}(\partial_t\phi + \DD \phi)dx dt\\ + \int_0^{\infty}\int_{\R^3}(\frac{|u|^2}{2} - \kappa\dvg u)u\cdot \nabla \phi dx dt
    \end{multline*}
\end{enumerate}
\end{definition}
An analogous definition is available for system \eqref{Toy-Mod2}. The proof of existence of these Leray solutions follows the same lines as the one of the so-called Lemari\'e-Rieusset (see e.g. \cite{Lemar02} Chapters 32 \& 33 or \cite{Ser14} Appendix B) and will be presented elsewhere.
\subsection{Main result}
Take $u_0 = (u_0^1,u_0^2,u_0^3)$ which is a $(-1)$-homogeneous vector field such that $u_0|_{\partial B_1} \in C^{\infty}(\partial B_1)$. In this case, one can steadily show that
\[
|\partial^{\alpha} u_0(x)| \leq \frac{C(\alpha,u_0)}{|x|^{1+|\alpha|}},\quad \forall \alpha\in \N^3.
\]
We have the following decay estimate
\begin{theorem}\label{Thm1.1}
Let $u_0$ as above, and $u$ be a scale invariant Leray solution to system \eqref{Toy-Mod1} or system \eqref{Toy-Mod2}. Then $U(\cdot) := u(\cdot,1)$, the solution profile at time $t=1$, belongs to $C^{\infty}(\R^3)$ and 
\[
|\partial^{\al}(U - S_{\kappa}(1)u_0)(x)| \leq \frac{C(\al,\kappa,u_0)}{(1 + |x|)^{3 + |\al|}},
\]
for all $\al\in \N^3$ (with $|\al| = \al_1 + \al_2 + \al_3$). Here $S_{\kappa}(t)$ denotes the semigroup associate to the time dependent Lam\'e system:
\[
\partial_t f - \DD f - \kappa\nabla \dvg f = 0.
\]
\end{theorem}
\section{\texorpdfstring{$\epsilon$}{TEXT}-Regularity}
We prove in this section an $\epsilon$-regularity criteria similar to that of Caffarelli-Kohn-Nirenberg (see \cite{Caff82}) for a generalised counterpart of system \eqref{Toy-Mod1}. Our setting is as follows:\\
Let $\mathcal{O}$ be an open subset of $\R^{3+1}$ and $a,b\in L_m(\mathcal{O})$  with $m>5$. A function $u$ is called \textit{suitable weak solution} to 
\begin{equation}\label{E2.1}
    \partial_t u - \DD u - \kappa \nabla \dvg u + u\cdot \nabla u + \frac{u}{2}\dvg u + \frac{a}{2}\cdot \nabla u + \dvg (u\otimes \frac{a}{2}) + \dvg(b\otimes u) - \frac{b}{2}\dvg u = 0, 
\end{equation}
in $\mathcal{O}$ if $u\in L_{2,\infty}(\mathcal{O})\cap W^{1,0}_2(\mathcal{O})$, satisfies \eqref{E2.1} in the sense of distribution in $\mathcal{O}$ and
\begin{multline}\label{E2.2}
    \partial_t \frac{|u|^2}{2} - \DD \frac{|u|^2}{2} - \kappa \dvg(u\dvg u) + |\nabla u|^2 + \kappa (\dvg u)^2 + \dvg\left( (u+a)\frac{|u|^2}{2} \right)\\ + u\cdot\dvg(b\otimes u) - \frac{1}{2} u\cdot b \dvg u \leq 0
\end{multline}
in the sense of distributions. Let us point out that the term $u\cdot \dvg(a\otimes u)$ is indeed a distribution and should be understood in the following way
\[
\langle u\cdot \dvg(a\otimes u),\phi \rangle = - \int_{\mathcal{O}}a_i u_j u_{i,j}\phi dx dt - \int_{\mathcal{O}}a_i u_j u_i \phi_{,j} dx dt,
\]
for all $\phi \in C^{\infty}_0(\mathcal{O})$. All the other terms in \eqref{E2.2} obviously make sense because of the energy class of $u$ and the integrability condition on $a,b$.\\
The main theorem of this section reads as follows 

\begin{theorem}[$\epsilon$-regularity criterion]\label{Thm2.1}
Let $u$ be a suitable weak solution to \eqref{E2.1} in $Q$ with $a,b\in L_{m}(Q)$ and $m>5$. Then there exists $\ee_0 = \ee_0(\kappa,m)>0$ with the following properties: if 
\begin{equation}\label{E2.3}
    \left(\avint_Q |u|^3 dz\right)^{\frac{1}{3}} + \left(\avint_Q |a|^m dz\right)^{\frac{1}{m}} + \left(\avint_Q |b|^m dz\right)^{\frac{1}{m}} \leq \ee_0,
\end{equation}
then $u$ is H\"older continuous in $\overline{Q(\frac{1}{2})}$ with exponent $\al = \al(m)\in(0,1)$ and 
\begin{equation}\label{E2.4}
    \|u\|_{C^{\al,\frac{\al}{2}}(\overline{Q(\frac{1}{2})})} \leq C(\ee_0,\kappa,m).
\end{equation}
\end{theorem}

\begin{remark}
The proof of this theorem we present here is inspired from the proof of a similar result in the case of the $3$D incompressible Navier-Stokes equations given in \cite{Lin98,Ser14}. The particular case $a=b=0$ was tackled in \cite{Rus12} (see Theorem 7.1). Finally, let us point out that the case $\kappa = 0$ and $a=b=0$ was tackled in \cite{Hou19} using a more direct method which could also be extended to the general case \eqref{E2.1} upon some minor changes.
\end{remark}

The following auxiliary results will be needed for the proof of Theorem \ref{Thm2.1}.
\begin{lemma}\label{L2.2}
Set $Q_T := \R^3 \times (0,T)$ ($T>0$); let $f\in L_{s,l}(Q_T)$ and $F\in L_{s,l}(Q_T;\R^{3\times 3})$ with $1<s,l<\infty$. There exist two functions $v$ and $w$ that uniquely solve the systems
\begin{equation*}
    \left\{
    \begin{gathered}
    \partial_t v -\DD v - \kappa \nabla \dvg v = f\quad\mbox{in }Q_T\\
    v|_{t=0} = 0\quad\mbox{in }\R^3
    \end{gathered}
    \right.\quad\mbox{and}\quad
    \left\{
    \begin{gathered}
    \partial_t w -\DD w - \kappa \nabla \dvg w = \dvg F\quad\mbox{in }Q_T\\
    w|_{t=0} = 0\quad\mbox{in }\R^3
    \end{gathered}
    \right.
\end{equation*}
such that
\begin{gather*}
\|\partial_t v\|_{L_{s,l}(Q_T)} + \|\nabla^2 v\|_{L_{s,l}(Q_T)} \leq c(\kappa,s,l)\|f\|_{L_{s,l}(Q_T)},\quad \|\nabla w\|_{L_{s,l}(Q_T)} \leq c(\kappa,s,l)\|F\|_{L_{s,l}(Q_T)};\\
\|\nabla v\|_{L_{l_1}(0,T;L_{s_1}(\R^3))} \leq c(\kappa,s,s_1,l,l_1)T^{\frac{1}{2}\left(1 - \frac{2}{l}-\frac{3}{s}+ \frac{2}{l_1} + \frac{3}{s_1} \right)}\|f\|_{L_{s,l}(Q_T)}~\forall T\geq 0,\\
\mbox{with $s\leq s_1\leq \infty,~l\leq l_1\leq \infty$ such that } 1 - \frac{2}{l}-\frac{3}{s} +\frac{2}{l_1} + \frac{3}{s_1}>0;\\
\|v\|_{L_{l_2}(0,T;L_{s_2}(\R^3))} \leq c(\kappa,s,s_2,l,l_2)T^{\frac{1}{2}\left(2 - \frac{2}{l}-\frac{3}{s}+ \frac{2}{l_2}+ \frac{3}{s_2} \right)}\|f\|_{L_{s,l}(Q_T)}~\forall T\geq 0,\\
\mbox{with $s\leq s_2\leq \infty,~l\leq l_2 \leq \infty$ such that } 2 - \frac{2}{l}-\frac{3}{s}+ \frac{2}{l_2} + \frac{3}{s_2}>0;\\
\|w\|_{L_{l_3}(0,T;L_{s_3}(\R^3))} \leq c(\kappa,s,s_3,l,l_3)T^{\frac{1}{2}\left(1 - \frac{2}{l}-\frac{3}{s}+\frac{2}{l_3} + \frac{3}{s_3} \right)}\|F\|_{L_{s,l}(Q_T)}~\forall T\geq 0,\\
\mbox{with $s\leq s_3\leq \infty,~l\leq l_3 \leq \infty$ such that } 1 - \frac{2}{l}-\frac{3}{s} + \frac{2}{l_3}+ \frac{3}{s_3}>0.
\end{gather*}
Finally, we have H\"older continuity of $v$ as soon as $\mu:= 2 - \frac{2}{l} - \frac{3}{s}>0$ and for $w$ when $\al := 1 - \frac{2}{l} - \frac{3}{s}>0$; to be more precise, we have the following estimates:
\[
|v(z_1) - v(z_2)|\leq c(\kappa,s,l)\left(|x_1 - x_2| + \sqrt{|t_1 - t_2|} \right)^{\mu}\|f\|_{L_{s,l}(Q_T)}
\]
and similarly 
\[
|w(z_1) - w(z_2)|\leq c(\kappa,s,l)\left(|x_1 - x_2| + \sqrt{|t_1 - t_2|} \right)^{\al}\|F\|_{L_{s,l}(Q_T)},
\]
for all $z_1=(x_1,t_1),z_2=(x_2,t_2)\in Q_T$.
 \end{lemma}

\begin{proof}[Sketch of proof]
Uniqueness is straightforward. For the existence part, we present only the proof for the function $w$ since the function $v$'s case follows the same ideas.\\ 
There exists a function $q$ (using the Newtonian representation for solutions of the Poisson equation together with singular integrals' theory) such that
\[ \DD q = \dvg\dvg(F), \]
and 
\[ \|q\|_{L_{s,l}(Q_T)} \leq c \|F\|_{L_{s,l}(Q_T)} \]
Next, we introduce the function $F_0 := \dvg (q I_3 - F)$ and let us notice that $\dvg F_0 = 0$ in $\mathcal{D}'(Q_T)$. From well-known solvability results for the heat equation (see e.g. \cite{Kry08,Lady68}), we have the existence of two functions $w^1$ and $w^2$ such that
\begin{equation}\label{E2.5}
    \left\{
    \begin{gathered}
    \partial_t w^1 -(1+\kappa)\DD w^1 = \nabla q\quad\mbox{in }Q_T\\
    w^1|_{t=0} = 0\quad\mbox{in }\R^3
    \end{gathered}
    \right.\quad\mbox{and}\quad
    \left\{
    \begin{gathered}
    \partial_t w^2 -\DD w^2 = F_0\quad\mbox{in }Q_T\\
    w^2|_{t=0} = 0\quad\mbox{in }\R^3
    \end{gathered}
    \right.
\end{equation}
Moreover, the following estimate is available:
\begin{gather*}
    \|\nabla w^1\|_{L_{s,l}(Q_T)} + \|\nabla w^2\|_{L_{s,l}(Q_T)} \leq c(\kappa,s,l)\|F\|_{L_{s,l}(Q_T)}.\\
    \|w^1\|_{L_{l_3}(0,T;L_{s_3}(\R^3))} + \|w^2\|_{L_{l_3}(0,T;L_{s_3}(\R^3))} \leq c(\kappa,s,s_3,l,l_3)T^{\frac{1}{2}\left(1 - \frac{2}{l}-\frac{3}{s} + \frac{3}{l_3} + \frac{3}{s_3} \right)}\|F\|_{L_{s,l}(Q_T)}\\
    \forall T\geq 0,\mbox{ with $s\leq s_3\leq \infty,~l\leq l_3\leq\infty$ such that } 1 - \frac{2}{l}-\frac{3}{s} + \frac{3}{l_3} + \frac{3}{s_3}>0.
\end{gather*}
The latter estimate comes from well-known properties of the volume heat potential but the former is a bit more subtle; we refer to \cite{Kry01} (Theorem 1.1) for a proof of this statement. Next, using Campanato's characterisation for H\"older continuous functions, one gets without too much difficulty (and with the help of Poincaré's inequality on balls) that
\[
|w^1(z_1) - w^1(z_2)| + |w^2(z_1) - w^2(z_2)|\leq c(\kappa,s,l)\left(|x_1 - x_2| + \sqrt{|t_1 - t_2|} \right)^{\al}\|F\|_{L_{s,l}(Q_T)},
\]
for all $z_1=(x_1,t_1),z_2=(x_2,t_2)\in Q_T$ as long as $\al>0$.\\
Finally, notice that $\curl w^1 = \dvg w^2 = 0$ in $Q_T$ and $\DD w^1 = \nabla \dvg w^1$ (at least in the sense of distributions); and we are done by setting $w := w^1 + w^2$.
\end{proof}
\begin{remark}
A suitable scaling argument in \eqref{E2.5} allow us to see how the constants in the above estimates depend on $\kappa$.
\end{remark}
Next, we have the following local regularity result for the time-dependent Lam\'e system.
\begin{lemma}[Local regularity]\label{L2.3}
Let $u\in L_2(Q)$ such that 
\[
\partial_t u - \DD u - \kappa \nabla \dvg u = 0\mbox{ in }\mathcal{D}'(Q). 
\]
Then, for any $k=0,1,2,\ldots$ and any $0<\varrho<1$, there exists $C = C(\kappa,\varrho,k)>0$ such that
\[
\sup_{(x,t)\in Q(\varrho)}|\nabla^k u(x,t)| \leq C \left(\avint_Q |u|^2 dz \right)^{\frac{1}{2}}.
\]
\end{lemma}

\begin{proof}
We see without too much difficulty that
\begin{equation}\label{E2.6}
    \int_{Q((1+\varrho)/2)}|\nabla u|^2 dz \leq c(\kappa,\varrho)\int_{Q}|u|^2dz,
\end{equation}
(for an arbitrary $0<\varrho<1$) and 
\[ 
\partial_t \curl u - \DD\curl u = 0\quad\mbox{and}\quad \partial_t\dvg u - (1+\kappa)\DD\dvg u = 0,
\]
in the sense of distributions. Next, from local well-known regularity results for the heat equation (see e.g. \cite{Kry08,Lady68})
\begin{gather*}
    \sup_{(x,t)\in Q((1+3\varrho)/4)}|\nabla^k \curl u(x,t)| \leq C(\varrho,k) \left(\int_{Q((1+\varrho)/2)} |\curl u|^2 dz \right)^{\frac{1}{2}}\\
    \sup_{(x,t)\in Q((1+3\varrho)/4)}|\nabla^k \dvg u(x,t)| \leq C(\kappa,\varrho,k) \left(\int_{Q((1+\varrho)/2)} |\dvg u|^2 dz \right)^{\frac{1}{2}} 
\end{gather*}
for any $k=0,1,2\ldots$ Thus, from \eqref{E2.6}, we have 
\begin{equation}\label{E2.7}
    \sup_{(x,t)\in Q((1+3\varrho)/4)}(|\nabla^k \curl u(x,t)| + |\nabla^k \dvg u(x,t)|) \leq c(\kappa,\varrho,k)\left(\int_{Q} |u|^2 dz \right)^{\frac{1}{2}},
\end{equation}
for any $k=0,1,2\ldots$ Finally, using the identity
\[
-\DD u = \curl(\curl u) - \nabla \dvg u \mbox{ in }\mathcal{D}'(Q), 
\]
together with the stationary analogue of the first estimate in Theorem 2.4.9 of \cite{Kry08} (for instance) and taking into account \eqref{E2.7}, we have that the lemma is proved. 
\end{proof}
Our next lemma gives us an estimate for a generalised time dependent Lam\'e system.
\begin{lemma}\label{L2.4}
Let $a,b\in L_m(Q;\R^3)$ and $F\in L_m(Q;\R^{3\times 3})$ such that
\begin{equation}\label{E2.8}
    \|a\|_{L_m(Q)} + \|b\|_{L_m(Q)} + \|F\|_{L_m(Q)} \leq M,
\end{equation}
for some arbitrary $M>0$ and $m>5$. Let also $\lbd\in \R^3$ with $|\lbd|\leq M$ and $u\in L_{2,\infty}(Q)\cap W^{1,0}(Q)$ such that
\begin{equation}\label{E2.9}
    \essup_{-1<t<0}\int_B |u(x,t)|^2 dx + \int_{-1}^0\int_B |\nabla u|^2 dx dt \leq M^2.
\end{equation}
Assume $u$ satisfies
\begin{equation*}
    \partial_t u - \DD u - \kappa \nabla \dvg u + \lbd\cdot \nabla u + \frac{\lbd}{2}\dvg u + \frac{a}{2}\cdot \nabla u + \dvg (u\otimes \frac{a}{2}) + \dvg(b\otimes u) - \frac{b}{2}\dvg u = \dvg F, 
\end{equation*}
in the sense of distributions. Then $u$ is H\"older continuous in $Q(\frac{1}{2})$ with exponent $\al = \al(m)\in (0,1)$ and 
\[
\|u\|_{C^{\al,\frac{\al}{2}}(\overline{Q(\frac{1}{2})})} \leq C(\kappa,m,M).
\]
\end{lemma}

\begin{proof}
Let $0<\varrho<r\leq 1$ and $0\leq \varphi_{\varrho,r}\in C^{\infty}_0(B)$ such that $\varphi_{\varrho,r}\equiv 1$ in $B((r+\varrho)/2)$ and $\varphi_{\varrho,r} \equiv 0$ in $B(r)\setminus B((3 r + \varrho)/4)$. Next, let us set
\[
F^{\varrho,r} := \varphi_{\varrho,r}\left(F - u\otimes\lbd - \frac{\lbd}{2}\otimes u - u\otimes \frac{a}{2} - b\otimes u\right),~f^{\varrho,r} := \left(\frac{b}{2}\dvg u - \frac{a}{2}\cdot \nabla u\right)\varphi_{\varrho,r};  
\]
If $u\in L_q(Q(r))$ and $\nabla u\in L_p(Q(r))$ ($p,q\geq 2$) with
\begin{equation}\label{E2.10}
    \frac{1}{m} + \frac{1}{q} < \frac{1}{p} + \frac{1}{2}\left(\frac{1}{m} - \frac{1}{5}\right)\quad\mbox{and}\quad \frac{1}{m} + \frac{1}{p} - \frac{2}{5} < \frac{1}{q} + \frac{1}{2}\left(\frac{1}{m} - \frac{1}{5}\right) 
\end{equation}
(e.g. \eqref{E2.10} holds true for $q=10/3$ and $p=2$) then 
\[
F^{\varrho,r}\in L_{\frac{mq}{m+q}}(Q(r))\mbox{ and }f^{\varrho,r}\in L_{\frac{mp}{m+p}}(Q(r)).
\]
From Lemma \ref{L2.2}, we have the existence of two functions $v^{\varrho,r}$ and $w^{\varrho,r}$ such that
\begin{equation*}
    \left\{
    \begin{gathered}
    \partial_t v^{\varrho,r} -\DD v^{\varrho,r} - \kappa \nabla \dvg v^{\varrho,r} = f^{\varrho,r}\quad\mbox{in }\R^3\times(-1,0)\\
    v^{\varrho,r}|_{t=-1} = 0\quad\mbox{in }\R^3
    \end{gathered}
    \right.
\end{equation*}
and
\begin{equation*}
    \left\{
    \begin{gathered}
    \partial_t w^{\varrho,r} -\DD w^{\varrho,r} - \kappa \nabla \dvg w^{\varrho,r} = \dvg F^{\varrho,r}\quad\mbox{in }\R^3\times(-1,0)\\
    w^{\varrho,r}|_{t=-1} = 0\quad\mbox{in }\R^3
    \end{gathered}
    \right.
\end{equation*}
Moreover the following estimates are available
\[
\|w^{\varrho,r}\|_{L_{\hat{q}}(\R^3\times (-1,0))} + \|\nabla w^{\varrho,r}\|_{L_{\frac{mq}{m+q}}(\R^3\times (-1,0))}  \leq c(m,p,q)\|F^{\varrho,r}\|_{L_{\frac{mq}{m+q}}(Q(r))},
\]
with \[ \frac{1}{\hat{q}} > \frac{1}{q} + \frac{1}{m} - \frac{1}{5}, \]
and
\[
\|v^{\varrho,r}\|_{L_{\tilde{q}}(\R^3\times(-1,0))} + \|\nabla v^{\varrho,r}\|_{L_{\tilde{p}}(\R^3\times(-1,0))} \leq c(m,p,q)\|f^{\varrho,r}\|_{L_{\frac{mp}{m+p}}(Q(r))},
\]
with 
\[
\frac{1}{\tilde{p}} > \frac{1}{p} + \frac{1}{m} - \frac{1}{5}\quad\mbox{and}\quad\frac{1}{\tilde{q}} > \frac{1}{p} + \frac{1}{m} - \frac{2}{5}.
\]
Now, we choose $q_1\geq 2$ and $p_1\geq 2$ such that
\begin{equation}\label{E2.11}
    \frac{1}{q_1} = \frac{1}{q} + \frac{1}{2}\left(\frac{1}{m} - \frac{1}{5}\right)\quad\mbox{and}\quad \frac{1}{p_1} = \frac{1}{p} + \frac{1}{2}\left(\frac{1}{m} - \frac{1}{5}\right);
\end{equation}
because of \eqref{E2.10}, we have that
\begin{gather}
    \begin{split}
        \|w^{\varrho,r}\|_{L_{q_1}(Q)} + \|\nabla w^{\varrho,r}\|_{L_{p_1}(Q)}  \leq c(\kappa,m,p,q)\|F^{\varrho,r}\|_{L_{\frac{mq}{m+q}}(Q(r))},\\
    \|v^{\varrho,r}\|_{L_{q_1}(Q)} + \|\nabla v^{\varrho,r}\|_{L_{p_1}(Q)}  \leq c(\kappa,m,p,q)\|f^{\varrho,r}\|_{L_{\frac{mp}{m+p}}(Q(r))};
    \end{split}
\end{gather}
moreover \eqref{E2.10} is true for $q$ and $p$ replaced respectively by $q_1$ and $p_1$.\\
Finally, let us notice that
\begin{equation}\label{E2.13}
    \partial_t(u - v^{\varrho,r} - w^{\varrho,r}) - \DD(u - v^{\varrho,r} - w^{\varrho,r}) - \kappa \nabla\dvg(u - v^{\varrho,r} - w^{\varrho,r}) = 0
\end{equation}
in $Q((r+\varrho)/2)$. which together with Lemma \ref{L2.3} lead to $u\in L_{q_1}(Q(\varrho))$ and $\nabla u \in L_{p_1}(Q(\varrho))$ where $q_1>q$ and $p_1>p$. The goal now is to iterate this process. We set $q_0 = 10/3$, $p_0 = 2$ and we define a sequence $(q_k)$ and $(p_k)$ via the following recursive formula
\[
\frac{1}{q_k} = \frac{1}{q_{k-1}} + \frac{1}{2}\left(\frac{1}{m} - \frac{1}{5}\right)\quad\mbox{and}\quad \frac{1}{p_k} = \frac{1}{p_{k-1}} + \frac{1}{2}\left(\frac{1}{m} - \frac{1}{5}\right)
\]
and we find $u\in L_{q_k}(Q(2^{-1} + 2^{-(1+k)}))$ and $\nabla u\in L_{p_k}(Q(2^{-1} + 2^{-(1+k)}))$ for all $k\geq 0$. Now, for a large enough $k_0 = k_0(m)$, we have that
\[
\frac{m q_{k_0}}{m+q_{k_0}} >5 \mbox{ and } \frac{m p_{k_0}}{m+p_{k_0}} >\frac{5}{2}
\]
thus for $\varrho = 2^{-1} + 2^{-(2+k_0)}$ and $r = 2^{-1} + 2^{-(1+k_0)}$, we find that there exists $\alpha = \alpha(m)\in (0,1)$ such that $v^{\varrho,r},w^{\varrho,r} \in C^{\alpha,\frac{\alpha}{2}}(Q(2^{-1} + 2^{-(2+k_0)}))$. Using one more time \eqref{E2.13} and Lemma \ref{L2.3}, we get that $u\in C^{\alpha,\frac{\alpha}{2}}(\overline{Q(1/2)})$. A careful track of the constants in the above process yields also the estimate claimed in the Lemma, we omit the details here. Thus the lemma is proved.
\end{proof}
We can return now to the proof of Theorem \ref{Thm2.1}. We first prove a so-called "oscillation lemma", which roughly speaking asserts that if $u$ is of \textit{small oscillation} in $Q$, then the oscillation is even smaller in $Q_{\theta}$ for $0<\theta<1$ (see e.g. \cite{Ser14,Jia14} where an analogue of this result was proved for the incompressible Navier-Stokes equations).

In order to state our lemma and for what follows, some additional notions are needed; we introduce
\[
Y(z_0,R;u) := \left( \avint_{Q(z_0,R)}|u-(u)_{z_0,R}|^3 dz\right)^{\frac{1}{3}},~Y_{\theta}(u) := Y(0,\theta;u)
\]

\begin{lemma}\label{L2.5}
Given any numbers $\theta\in(0,1/3)$, $m>5$ and $M>0$, there are three constants $\al = \al(m)\in(0,1)$, $\epsilon = \epsilon(\theta,m,M)>0$ and $C_1=C_1(\kappa,m,M)>0$ such that for any $\|a\|_{L_m(Q)}\leq M,~\|b\|_{L_m(Q)}\leq c_0$, $c_0>0$ being a small absolute constant (to be specified later), and any suitable weak solution $u$ to \eqref{E2.1} in $Q$, satisfying the additional conditions
\[
|(u)_{,1}| \leq M,\quad Y_1(u) + |(u)_{,1}|\left(\avint_Q |a|^m dz\right)^{\frac{1}{m}} + |(u)_{,1}|\left(\avint_Q |b|^m dz\right)^{\frac{1}{m}} <\ee,
\]
the following estimate is valid:
\[ Y_{\theta}(u) \leq C_1 \theta^{\al}\left[Y_1(u) + |(u)_{,1}|\left(\avint_Q |a|^m dz\right)^{\frac{1}{m}} + |(u)_{,1}|\left(\avint_Q |b|^m dz\right)^{\frac{1}{m}} \right]. \]
\end{lemma}

\begin{proof}
Assume that the statement is false. This means that there exist numbers $\theta\in(0,1/3)$, $m>5$ and $M>0$ and sequences $a^k$, $b^k$ and a sequence of suitable weak solutions $u^k$ to \eqref{E2.1} (with $a$ and $b$ replaced respectively by $a^k$ and $b^k$) such that
\begin{equation*}
    \begin{gathered}
    |(u^k)_1| \leq M,\quad\|a^k\|_{L_m(Q)}\leq M,~ \|b^k\|_{L_m(Q)} \leq c_0,\\
        Y_1(u^k) + |(u^k)_{,1}|\left(\avint_Q |a^k|^m dz\right)^{\frac{1}{m}} + |(u^k)_{,1}|\left(\avint_Q |b^k|^m dz\right)^{\frac{1}{m}} = \ee_k\to 0^+\mbox{ as }k\to\infty,\\
        Y_{\theta}(u^k) > C_1\theta^{\al}\ee_k,
    \end{gathered}
\end{equation*}
for all $k\in \N$. Next, we introduce the following functions 
\[
v^k := \frac{u^k - (u^k)_{,1}}{\ee_k},\quad -F^k := \frac{(u^k)_{,1}\otimes a^k}{2\ee_k} + \frac{b^k\otimes (u^k)_{,1}}{\ee_k};
\]
we have 
\begin{equation}\label{E2.14}
    \left(\avint_Q |v^k|^3 dz \right)^{\frac{1}{3}} + \left(\avint_Q |F^k|^m dz \right)^{\frac{1}{m}} \leq 4\quad\mbox{and }Y_{\theta}(v^k) > C_1 \theta^{\al}.
\end{equation}
Moreover,
\begin{multline}\label{E2.15}
    \partial_t v^k - \DD v^k - \kappa \nabla \dvg v^k + \ee_k\left(v^k\cdot\nabla v^k + \frac{v^k}{2}\dvg v^k \right) + (u^k)_{,1}\cdot \nabla v^k + \frac{(u^k)_{,1}}{2}\dvg v^k\\ + \frac{a^k}{2}\cdot \nabla v^k + \dvg(v^k \otimes \frac{a^k}{2}) + \dvg(b^k\otimes v^k) - \frac{b^k}{2}\dvg v^k = \dvg(F^k),
\end{multline}
in the sense of distributions on $Q$ and 
\begin{multline}\label{E2.16}
    \partial_t \frac{|v^k|^2}{2} - \DD \frac{|v^k|^2}{2} - \kappa\dvg(v^k\dvg v^k) + |\nabla v^k|^2 + \kappa(\dvg v^k)^2 + \dvg\left((\ee_k v^k + a^k + (u^k)_{,1})\frac{|v^k|^2}{2} \right)\\ + v^k\cdot\dvg((u^k)_{,1}\otimes \frac{v^k}{2} + b^k\otimes v^k - F^k) - \frac{1}{2} v^k\cdot b^k \dvg v^k \leq 0,
\end{multline}
in the sense of distributions on $Q$. From the previous inequation, we get that
\begin{multline*}
    \frac{1}{2}\int_B|v^k(x,t)|^2 \phi(x,t)dx + \int_{-1}^t\int_B|\nabla v^k|^2\phi dz + \kappa\int_{-1}^t\int_B (\dvg v^k)^2\phi dz\\ \leq \int_{-1}\int_B \frac{|v^k|^2}{2}(\partial_t \phi + \DD \phi)dz + \int_{-1}^t\int_B\frac{|v^k|^2}{2}(\ee_k v^k + a^k + (u^k)_{,1})\cdot\nabla \phi dz\\ + \int_{-1}^t\int_B\phi \nabla v^k:((u^k)_{,1}\otimes \frac{v^k}{2} + b^k\otimes v^k - F^k) dz + \int_{-1}^t\int_B v^k\otimes\nabla \phi : ((u^k)_{,1}\otimes \frac{v^k}{2} + b^k\otimes v^k - F^k) dz\\ - \kappa\int_{-1}^t\int_Bv^k\cdot \nabla \phi \dvg v^k dz + \frac{1}{2}\int_{-1}\int_B\phi v^k\cdot b^k \dvg v^k dz,
\end{multline*}
for any $0\leq \phi \in C^{\infty}_0(B\times(-1,1))$. Next, we define
\[
E_k(r) := \frac{1}{2}\essup_{-r^2<t<0}\int_{B(r)}|v^k(x,t)|^2 dx + \int_{-r^2}\int_{B(r)}[|\nabla v^k(x,t)|^2 +\kappa (\dvg v^k(x,t))^2] dz, 
\]
for any $0<r<1$. Our goal now is to get an uniform estimate (in $k$) for $E_k$. To this end we start by recalling the following well-known multiplicative inequality
\[
\|v^k\|^2_{L_{\frac{10}{3}}(Q(r))} \leq C E_k(r),
\]
with $C>0$ being an universal constant. Then, for any $1/2<r_1<r_2\leq 1$, if we choose appropriately the test function $\phi$ in the above local energy inequality together with the help of H\"older's inequality and the estimates on $v^k,~a^k,~b^k$ and $F^k$, we have that:
\begin{align*}
    E_k(r_1) &\leq \frac{C}{(r_2 - r_1)^2} + \frac{C}{r_2 - r_1}\left\{ \ee_k\int_Q|v^k|^3dz + \left(\int_Q|v^k|^3 dz\right)^{\frac{2}{3}}\left(\int_Q|a^k|^3 dz\right)^{\frac{1}{3}}\right.\\&\mathrel{\phantom{=}}\left. + |(v^k)_{,1}|\int_Q|v^k|^2dz \right\} + E_k(r_2)^{\frac{1}{2}}\left\{ |(u^k)_{,1}|\left(\int_Q|v^k|^2 dz\right)^{\frac{1}{2}}\right.\\&\mathrel{\phantom{=}}\left. + \left(\int_Q|b^k|^5 dz\right)^{\frac{1}{5}}\left(\int_{Q(r_2)}|v^k|^{\frac{10}{3}} dz\right)^{\frac{3}{10}} + \left( \int_Q|F^k|^2 dz \right)^{\frac{1}{2}} \right\}\\&\mathrel{\phantom{=}}+\frac{C}{r_2 - r_1}\left\{ |(u^k)_{,1}|\int_{Q}|v^k|^2 dz + \left(\int_Q|v^k|^3 dz\right)^{\frac{2}{3}}\left(\int_Q|a^k|^3 dz\right)^{\frac{1}{3}}\right.\\&\mathrel{\phantom{=}}\left. \left( \int_Q|v^k|^2 dz\right)^{\frac{1}{2}}\left( \int_Q|F^k|^2 dz\right)^{\frac{1}{2}} \right\} + E_k(r_2)^{\frac{1}{2}}\left\{\frac{C(\kappa)}{(r_2 - r_1)^2}\int_Q|v^k|^2 dz \right\}^{\frac{1}{2}}\\
    &\leq \frac{C(\kappa,M)}{(r_2 - r_1)^2} + C(M)E_k(r_2)^{\frac{1}{2}} + (\frac{1}{4} + C\|b^k\|_{L_m(Q)})E_k(r_2)\\
    &\leq \frac{C(\kappa,M)}{(r_2 - r_1)^2} + (\frac{1}{2} + C\|b^k\|_{L_m(Q)})E_k(r_2).
\end{align*}
Note that $\|b\|_{L_m(Q)}\leq c_0$ with $c_0$ small. Therefore, if we choose $c_0$ such that $Cc_0<1/2$, we can iterate the above estimate (see e.g. \cite{Evans86} Lemma 5.2) and conclude that
\begin{equation}\label{E2.17}
    E_k(3/4) \leq C(\kappa,M),~\forall k\in \N.
\end{equation}
This together with the fact that $v^k$ satisfies \eqref{E2.15} yield
\begin{equation}\label{E2.18}
    \|\partial_t v^k\|_{L_{4/3}(-(3/4)^2,0;H^{-1}(B(3/4)))} \leq C(\kappa,M),~\forall k\in \N,
\end{equation}
where $H^{-1}(B(3/4))$ stands here for the dual of the Sobolev space $H^1_0(B(3/4))$. Now, from Aubin-Lions and Banach-Alaoglu compactness results, we can choose subsequences of $v^k,~a^k,~b^k,~(u^k)_{,1}$ and $F^k$ (which we still denote $v^k,~a^k,~b^k,~(u^k)_{,1}$ and $F^k$) such that for some $\lbd \in \R,~v\in L_3(Q(3/4)),~a,b,F\in L_m(Q(3/4))$, we have
\begin{equation*}
    \begin{gathered}
    v^k \to v \mbox{ strongly in } L_3(Q(3/4)),~ (u^k) \to \lbd\\
    a^k \rightharpoonup a,~b^k \rightharpoonup b,~F^k \rightharpoonup F\quad \mbox{weakly in } L_m(Q(3/4)). 
    \end{gathered}
\end{equation*}
Moreover, we have that
\begin{equation}
    \begin{gathered}
    \left(\int_{Q(3/4)}(|a| + |b| + |F|)^m dz \right)^{\frac{1}{m}}\leq C(M),\quad |\lbd |\leq M,\\
    \essup_{-(3/4)^2<t<0}\int_{B(3/4)}|v(x,t)|^2 dx + \int_{Q(3/4)}|\nabla v(x,t)|^2 dz \leq C(\kappa,M).
    \end{gathered}
\end{equation}
Finally, from \eqref{E2.15}, we see that
\begin{equation*}
    \partial_t v - \DD v - \kappa \nabla \dvg v + \lbd\cdot \nabla v + \frac{\lbd}{2}\dvg v + \frac{a}{2}\cdot \nabla v + \dvg (v\otimes \frac{a}{2}) + \dvg(b\otimes v) - \frac{b}{2}\dvg v = \dvg F,
\end{equation*}
in the sense of distributions in $Q(3/4)$. Thus, from Lemma \ref{L2.4}, we have that 
\[
\|v\|_{C^{\alpha,\frac{\alpha}{2}}(\overline{B(1/2)})} \leq C(\kappa,m,M),
\]
for some $\al = \al(m)\in (0,1)$. This implies that 
\begin{equation}\label{E2.20}
    Y_{\theta}(v) \leq C(\kappa,m,M)\theta^{\al};
\end{equation}
but because of \eqref{E2.14} and the strong $L_3$-convergence of $v^k$, we also have that $Y_{\theta}(v)\geq C_1 \theta^{\alpha}$. This together with \eqref{E2.20} give us
\[
C_1 \leq C(\kappa,m,M).
\]
If from the beginning, $C_1$ is chosen so that $C_1>2 C(\kappa,m,M)$, we arrive at a contradiction and the Lemma is proved. 
\end{proof}

Lemma \ref{L2.5} admits the following iterations.
\begin{lemma}\label{L2.6}
Given numbers $M>0$ and $m>5$, we choose $\theta \in (0,1/3)$ so that 
\begin{equation}
    C_1\theta^{\al - \beta} < 1\quad\mbox{and}\quad\theta < c_1
\end{equation}
where $c_1=c_1(m)>0$ is a small number to be specified later, $C_1,\al$ are as in Lemma \ref{L2.5} and $\beta := \al/2$. Then, there exists $\ee_* = \ee_*(\kappa,\theta,m,M)<\ee$ sufficiently small, such that for any $\|a\|_{L_m(Q)}\leq M,~\|b\|_{L_m(Q)}\leq c_0$ ($\ee = \ee(\theta,m,M)>0$ and $c_0>0$ being also as in Lemma \ref{L2.5}), and any suitable weak solution $u$ to \eqref{E2.1} in $Q$, satisfying the additional conditions
\[
|(u)_{,1}| \leq M/2,\quad Y_1(u) + M\left(\avint_Q |a|^m dz\right)^{\frac{1}{m}} + M\left(\avint_Q |b|^m dz\right)^{\frac{1}{m}} <\ee_*,
\]
the following holds: $\forall k=1,2,\ldots,$ we have 
\begin{equation*}
    |(u)_{,\theta^{k-1}}| \leq M,
\end{equation*}
\begin{equation*}
     Y_{\theta^{k-1}}(u) + \theta^{k-1}|(u)_{,\theta^{k-1}}|\left[\left(\avint_{Q(\theta^{k-1})} |a|^m dz\right)^{\frac{1}{m}} + \left(\avint_{Q(\theta^{k-1})} |b|^m dz\right)^{\frac{1}{m}}\right] <\ee_*,
\end{equation*}
\begin{equation*}
    Y_{\theta^k}(u) \leq \theta^{\beta}\left\{ Y_{\theta^{k-1}}(u) + \theta^{k-1}|(u)_{,\theta^{k-1}}|\left[\left(\avint_{Q(\theta^{k-1})} |a|^m dz\right)^{\frac{1}{m}} + \left(\avint_{Q(\theta^{k-1})} |b|^m dz\right)^{\frac{1}{m}}\right] \right\}.
\end{equation*}
\end{lemma}  

\begin{proof}
We prove the lemma by induction; the case $k=1$ is true thanks to Lemma \ref{L2.5}. Now, suppose the conclusion is true for $k\leq k_0$ ($k_0\geq 1$) and let us show that it remains the case for $k=k_0 + 1$.\\
For all $k\leq k_0$, we have
\begin{align*}
    Y_{\theta^k}(u) &\leq \theta^{\beta}\left\{ Y_{\theta^{k-1}}(u) + \theta^{k-1}|(u)_{,\theta^{k-1}}|\left[\left(\avint_{Q(\theta^{k-1})} |a|^m dz\right)^{\frac{1}{m}} + \left(\avint_{Q(\theta^{k-1})} |b|^m dz\right)^{\frac{1}{m}}\right] \right\}\\
    &\leq \theta^{\beta}\left\{ Y_{\theta^{k-1}}(u) + \theta^{(k-1)(1-5/m)}M\left[\left(\avint_{Q} |a|^m dz\right)^{\frac{1}{m}} + \left(\avint_{Q} |b|^m dz\right)^{\frac{1}{m}}\right] \right\}\\
    &\leq \theta^{\beta} Y_{\theta^{k-1}}(u) + \theta^{k \beta_1}\ee_*,
\end{align*}
with $\beta_1 = \min\{\beta,1-5/m\}$. By iterating the last inequality, we get
\[
Y_{\theta^k}(u) \leq \theta^{k\beta} Y_1(u) + k\theta^{k \beta_1}\ee_*,\quad\forall k\leq k_0.
\]
Thus,
\begin{align*}
    |(u)_{,\theta^{k_0}}| &\leq \sum_{k=1}^{k_0}|(u)_{,\theta^{k}} - (u)_{,\theta^{k-1}}| + |(u)_{,1}|\\
    &\leq \theta^{-5/3}\sum_{k=1}^{k_0}Y_{\theta^{k-1}}(u) + |(u)_{,1}|\\
    &\leq \theta^{-5/3}\sum_{k=1}^{k_0}\left(\theta^{(k-1)\beta} + (k-1)\theta^{(k-1)\beta_1} \right)\ee_* + M/2\\
    &\leq \underbrace{\theta^{-5/3}\left( (1-\theta^{\beta})^{-1} + \sum_{k=0}^{\infty}k\theta^{k\beta_1} \right)}_{C(\theta,m)}\ee_* + M/2.
\end{align*}
By choosing $\ee_* \leq M(2C(\theta,m))^{-1}$, we find that
\[ |(u)_{,\theta^{k_0}}| \leq M. \]
Moreover,
\begin{align*}
    Y_{\theta^{k_0}}(u) + \theta^{k_0}|(u)_{,\theta^{k_0}}|\left[\left(\avint_{Q(\theta^{k_0})} |a|^m dz\right)^{\frac{1}{m}} + \left(\avint_{Q(\theta^{k_0})} |b|^m dz\right)^{\frac{1}{m}}\right]&\leq \theta^{\beta}\ee_* + \theta^{(1-5/m)k_0}\ee_*\\
    &\leq \theta^{\beta}\ee_* + \theta^{(1-5/m)}\ee_*\\
    &<\ee_*,
\end{align*}
if we choose $\theta < c_1(m)$ to be small enough. Now, set
\begin{equation*}
    \begin{gathered}
    u^0(y,s) = \theta^{k_0}u(\theta^{k_0} y,\theta^{2 k_0} s),\\
    a^0(y,s) = \theta^{k_0}a(\theta^{k_0} y,\theta^{2 k_0} s),\quad b^0(y,s) = \theta^{k_0}b(\theta^{k_0} y,\theta^{2 k_0} s),
    \end{gathered}
\end{equation*}
with $(y,s)\in Q$. We show steadily that $u^0$ is also a suitable weak solution to \eqref{E2.1} with $a$ and $b$ replaced respectively by $a^0$ and $b^0$; moreover the conditions stated in Lemma \ref{L2.5} are satisfied for these new functions. Consequently, we have (thanks to Lemma \ref{L2.5})
\[
Y_{\theta}(u^0) \leq \theta^{\beta}\left[Y_1(u^0) + |(u^0)_{,1}|\left(\avint_Q |a^0|^m dz\right)^{\frac{1}{m}} + |(u^0)_{,1}|\left(\avint_Q |b^0|^m dz\right)^{\frac{1}{m}} \right],
\]
that is 
\[
Y_{\theta^{k_0 + 1}}(u) \leq \theta^{\beta}\left\{ Y_{\theta^{k_0}}(u) + \theta^{k_0}|(u)_{,\theta^{k_0}}|\left[\left(\avint_{Q(\theta^{k_0})} |a|^m dz\right)^{\frac{1}{m}} + \left(\avint_{Q(\theta^{k_0})} |b|^m dz\right)^{\frac{1}{m}}\right] \right\}.
\]
This concludes the induction and the proof of the lemma.
\end{proof}
By translation and dilatation, we obtain the following corollary.
\begin{corollary}\label{Cor2.6.1}
Let $u$ be a suitable weak solution to \eqref{E2.1} in $Q(z_0,R)$ with
\[
R^{1-5/m}\|a\|_{L_m(Q(z_0,R))}\leq M\quad\mbox{and}\quad R^{1-5/m}\|b\|_{L_m(Q(z_0,R))}\leq c_0
\]
with $c_0$ as in Lemma \ref{L2.5}. Take $\theta,\beta$ and $\ee_*$ as in the previous lemma. If we have in addition that
\[
R|(u)_{z_0,R}| \leq M/2 \mbox{ and } R Y(z_0,R;u) + R M\left[\left(\avint_{Q(z_0,R)} |a|^m dz\right)^{\frac{1}{m}} + \left(\avint_{Q(z_0,R)} |b|^m dz\right)^{\frac{1}{m}}\right] <\ee_*,
\]
then for all $k\geq 1$
\[
R|(u)_{Q(z_0,\theta^{k-1}R)}| \leq M,
\]
and
\begin{multline*}
    Y(z_0,\theta^k R;u) \leq \theta^{\beta}\left[ R Y(z_0,\theta^{k-1}R;u) + R\theta^{k-1}|(u)_{z_0,\theta^{k-1}R}|\left(\avint_{Q(z_0,\theta^{k-1}R)} |a|^m dz\right)^{\frac{1}{m}}\right.\\ + \left. R\theta^{k-1}|(u)_{z_0,\theta^{k-1}R}|\left(\avint_{Q(z_0,\theta^{k-1}R)} |b|^m dz\right)^{\frac{1}{m}}\right].
\end{multline*}
\end{corollary}

We are now ready to prove Theorem \ref{Thm2.1}.
\begin{proof}[Proof of Theorem \ref{Thm2.1}]
Let $M=2020$ and choose $\theta$ according to Lemma \ref{L2.6}. Define 
\[
A := \left(\avint_Q|u|^3 dz\right)^{\frac{1}{3}} + \left(\avint_Q|a|^m dz\right)^{\frac{1}{m}} + \left(\avint_Q|b|^m dz\right)^{\frac{1}{m}}.
\]
Observe that
\[ Q(z_0,1/4) \subset Q \mbox{ if } z_0\in Q(1/2), \]
\[ 
\frac{1}{4}|(u)_{z_0,1/4}| \leq c A,\quad \left(\frac{1}{4}\right)^{1-5/m}(\|a\|_{L_m(Q(z_0,1/4))}+\|b\|_{L_m(Q(z_0,1/4))}) \leq c(m)A,
\]
and 
\[
\frac{1}{4}Y(z_0,1/4;u) + \frac{1}{4}\times M \left[  \left(\avint_{Q(z_0,1/4)}|a|^m dz\right)^{\frac{1}{m}} + \left(\avint_{Q(z_0,1/4)}|b|^m dz\right)^{\frac{1}{m}} \right] \leq c(m) A;
\]
Now, we choose
\[
\ee_0 < \min\{ 2020/c, 2020/c(m), c_0/c(m), \ee_*/c(m) \}
\]
with $c_0$ and $\ee_*$ as in Lemma \ref{L2.6}. Consequently, by applying Corollary \ref{Cor2.6.1}, we are able to prove (using a similar iteration process we used in the proof of Lemma \ref{L2.6}) that
\[
Y(z_0,\theta^k/4;u) \leq C(\kappa,\theta,m,M)\theta^{k\beta_2},
\]
for all $z_0\in Q(1/2)$, $k\geq 1$ and with $\beta_2 = 1/2(\beta+\beta_1)$. H\"older continuity of $u$ in $Q(\overline{1/2})$ follows from Campanato's type condition. The theorem is proved.
\end{proof}
Theorem \ref{Thm2.1} can be strengthen in the following manner, by removing the smallness condition on $a$ and $b$.
\begin{theorem}[Improved $\ee$-regularity criterion]\label{Thm2.7}
Let $u$ be a suitable weak solution to \eqref{E2.1} in $Q$ with $\|a\|_{L_m(Q)} + \|b\|_{L_m(Q)} \leq M$ for some $M>0$ and $m>5$. Then there exists $\ee_1 = \ee_1(\kappa,m,M)>0$ with the following properties: if 
\begin{equation}
    \left(\avint_Q |u|^3 dz\right)^{\frac{1}{3}} \leq \ee_1,
\end{equation}
then $u$ is H\"older continuous in $\overline{Q(\frac{1}{2})}$ with exponent $\al = \al(m)\in(0,1)$ and 
\begin{equation}
    \|u\|_{C^{\al,\frac{\al}{2}}(\overline{Q(\frac{1}{2})})} \leq C(\ee_1,\kappa,m).
\end{equation}
\end{theorem}
We skip the proof of Theorem \ref{Thm2.7} since it is essentially a repetition of the proof of an analogous result for the incompressible Navier-Stokes system in \cite{Jia14} (see Theorem 2.2 in there).
\section{Local in space near initial time smoothness of Leray solutions}
In this section, we use the $\ee$-regularity theorem(s) proved in the previous section to study the local in space near the initial time smoothness of the Leray solutions (or energy solutions) to our model \eqref{Toy-Mod1}.\\
For our future analysis, the following a priori estimate will be needed. It was first proved in \cite{Lemar02} for the incompressible Navier-Stokes system; here we follow the proof given in \cite{Jia13} (see Lemma 2.2) which is much simpler than the original one.
\begin{lemma}[A priori estimate for Leray solutions]\label{L3.1}
Let $u_0 \in L_{2,loc}(\R^3)$ such that for some $R>0$, $\al(R) := \sup_{x_0 \in \R^3}\int_{B(x_0,R)}|u_0(x)|^2 dx < \infty$ and let $u$ be a Leray solution to system \eqref{Toy-Mod1} with initial data $u_0$. Then, there exists some small absolute number $\mu>0$ such that for $0< \lbd < \mu \min\{1,\al(R)^{-2}R^2,(1+\kappa)^{-1}\}$, we have 
\begin{multline*}
    \sup_{x_0\in \R^3} \essup_{0<t<\lbd R^2}\int_{B(x_0,R)}\frac{|u(x,t)|^2}{2}dx \\ + \sup_{x_0\in \R^3}\int_0^{\lbd R^2}\int_{B(x_0,R)}\left(|\nabla u(x,t)|^2 + \kappa (\dvg u(x,t))^2\right) dx dt \leq C_0 \al(R),
\end{multline*}
with $C_0$ an absolute large constant.
\end{lemma}
\begin{proof}
Let $0\leq \phi\in C^{\infty}_0(B)$ such that $\phi \equiv 1$ in $B(1/2)$ and $\phi\equiv 0 \in B\setminus B(3/4)$; let $x_0 \in \R^3$, $R>0$ and set $\phi_{x_0,2R}(x) = \phi((x-x_0)/2 R)$. We have, from the local energy inequality verified by $u$, that: 
\begin{multline*}
    \int_{B(x_0,2R)}\frac{|u(x,t)|^2}{2}\phi_{x_0,2R}(x)dx \\+ \int_0^t\int_{B(x_0,2R)}\left(|\nabla u|^2 + \kappa(\dvg u)^2 \right)\phi_{x_0,2R} dx ds \leq \int_{B(x_0,2R)}\frac{|u_0(x)|^2}{2}\phi_{x_0,2R}(x) dx\\ + \int_0^t \int_{B(x_0,2R)}\frac{|u|^2}{2}\DD \phi_{x_0,2R} dx ds + \int_0^t \int_{B(x_0,2R)}(\frac{|u|^2}{2} - \kappa \dvg u)u\cdot \nabla\phi_{x_0,2R} dx ds.
\end{multline*}
Now, set
\begin{multline*}
    A_R(\lbd) = \sup_{x_0\in \R^3} \essup_{0<t<\lbd R^2}\int_{B(x_0,R)}\frac{|u(x,t)|^2}{2}dx \\ + \sup_{x_0\in \R^3}\int_0^{\lbd R^2}\int_{B(x_0,R)}\left(|\nabla u(x,t)|^2 + \kappa (\dvg u(x,t))^2\right) dx dt\quad (\lbd >0);
\end{multline*}
For a.e. $t\in (0,\lbd R^2)$, we have 
\begin{multline*}
    \int_{B(x_0,R)}\frac{|u(x,t)|^2}{2}dx + \int_0^t\int_{B(x_0,R)}|\nabla u|^2dx ds\\ + \int_0^t\int_{B(x_0,2R)}\kappa(\dvg u)^2\phi_{x_0,2R} dx ds \leq C\left( \al(R) + \lbd A_R(\lbd) 
    +\frac{1}{R}\int_0^{\lbd R^2}\int_{B(x_0,2R)}|u|^3 dx ds\right.\\ \left. +
     \int_0^{\lbd R^2}\int_{B(x_0,2R)}\kappa|\dvg u|\phi_{x_0,2R}^{\frac{1}{2}}|u||\nabla \phi_{x_0,2R}^{\frac{1}{2}}| dx ds \right).
\end{multline*}
Next, by known multiplicative inequality, we have
\[
\int_0^{\lbd R^2}\int_{B(x_0,2R)}|u|^3 dx ds \leq C\lbd^{\frac{1}{4}}R^{\frac{1}{2}}A_R(\lbd)^{\frac{3}{2}}\quad (\mbox{if }\lbd \leq 1),
\]
and thanks to Young's inequality that
\begin{multline*}
    \int_0^{\lbd R^2}\int_{B(x_0,2R)}\kappa|\dvg u|\phi_{x_0,2R}^{\frac{1}{2}}|u||\nabla \phi_{x_0,2R}^{\frac{1}{2}}|dx ds \leq \frac{1}{2} \int_0^t\int_{B(x_0,2R)}\kappa(\dvg u)^2\phi_{x_0,2R} dx ds\\ + C \kappa \lbd A_R(\lbd).
\end{multline*}
Consequently, we get that
\begin{multline*}
    \int_{B(x_0,R)}\frac{|u(x,t)|^2}{2}dx + \int_0^t\int_{B(x_0,R)}|\nabla u|^2dx ds\\ + \frac{1}{2}\int_0^t\int_{B(x_0,2R)}\kappa(\dvg u)^2\phi_{x_0,2R} dx ds \leq C \left[ \al(R) + (1+\kappa)\lbd A_R(\lbd) + \lbd^{\frac{1}{4}}R^{-\frac{1}{2}}A_R(\lbd)^{\frac{3}{2}}\right],
\end{multline*}
for a.e. $t\in (0,\lbd R^2)$ and all $x_0\in R^3$. Therefore
\[
A_R(\lbd) \leq C \left[ \al(R) + (1+\kappa)\lbd A_R(\lbd) + \lbd^{\frac{1}{4}}R^{-\frac{1}{2}}A_R(\lbd)^{\frac{3}{2}}\right];
\]
By choosing $\lbd \leq \min\{ 1,(2C(1+\kappa))^{-1} \}$, we find that
\[
A_R(\lbd) \leq 2 C \left( \al(R) + \lbd^{\frac{1}{4}}R^{-\frac{1}{2}}A_R(\lbd)^{\frac{3}{2}}\right),
\]
and from there the conclusion follows by standard continuation arguments. 
\end{proof}

Now we can prove the first important result of this section.
\begin{theorem}\label{Thm3.2}
Let $u_0 \in L_{2,loc}(\R^3)$ such that $\al:= \sup_{x_0 \in \R^3}\int_B |u_0(x)|^2 dx < \infty$. Suppose in addition that $M:=\|u_0\|_{L_m(B)}<\infty$ with $m>3$; Let us decompose $u_0 = u^1_0 + u^2_0$ with $u_0^1|_{B(4/3)} = u_0$, supp$u_0^1\subset\subset B(2)$ and $\|u_0^1\|_{L_m(\R^3)}\leq M$. Now, let $a$ be the locally in time defined mild solution to system \eqref{Toy-Mod1} with initial data $u_0^1$. Then, there exists a time $T=T(\al,\kappa,m,M)>0$ such that any Leray solution $u$ to \eqref{Toy-Mod1} satisfies:
\[
\|u-a\|_{C^{\gamma,\frac{\gamma}{2}}(\overline{B(1/2)\times [0,T]})} \leq C(\al,\kappa,m,M),
\]
for some $\gamma = \gamma(m)\in (0,1)$.
\end{theorem}
\begin{proof}
Let us start by discussing the decomposition of $u_0$ in the statement of the theorem; introduce the cut-off function $0\leq \varphi \in C^{\infty}_0(B)$ such that $\varphi \equiv 1$ in $B(4/3)$ and $\varphi\equiv 0$ in $B(2)\setminus B(3/2)$. We have the required splitting if we set $u_0^1 := u_0\varphi$ and $u_0^2 := u_0(1-\varphi)$.\\
Next, by assumption $a$ solves the Cauchy problem for system \eqref{Toy-Mod1} with initial data $u_0^1$ in $\R^3 \times [0,T_1]$, where $T_1 = T_1(\kappa,m,M)>0$. Such construction can be done by mimicking the one done for the incompressible Navier-Stokes equations in the Appendix of \cite{Esc03} (see Theorem 7.4). We present the details elsewhere. Their arguments also allow us to get that 
\begin{equation}
    \begin{gathered}
    \|a\|_{C([0,T_1];L_2(\R^3))} + \|\nabla a\|_{L_2(\R^3\times(0,T_1))} \leq C(m,M),\\
    \|a\|_{L_{\infty}(0,T_1;L_m(\R^3))} + \|a\|_{L_{\frac{5 m}{3}}(\R^3\times(0,T_1))} \leq C(\kappa,m,M).
    \end{gathered}
\end{equation}
Now, we set $v:= u - a$ and we observe that
\[
\partial_t v -\DD v - \kappa\nabla\dvg v + v\cdot \nabla v + \frac{v}{2}\dvg v + \frac{a}{2}\cdot \nabla v + \dvg(v\otimes \frac{a}{2}) + \dvg(a\otimes v) - \frac{a}{2}\dvg v = 0,  
\]
in the sense of distributions in $\R^3\times (0,T_1)$. Moreover, because $u$ and $a$ satisfy a local energy inequality, we see that
\begin{multline}\label{E3.2}
    \partial_t \frac{|v|^2}{2} - \DD \frac{|v|^2}{2} - \kappa\dvg(v\dvg v) + |\nabla v|^2 + \kappa (\dvg v)^2 + \dvg\left((v+a)\frac{|v|^2}{2}\right)\\ + v\cdot\dvg(a\otimes v) - \frac{1}{2}v\cdot a \dvg v \leq 0 \quad\mbox{in }\mathcal{D}'(\R^3\times (0,T_1)).
\end{multline}
In other words, $v$ is a suitable weak solution to \eqref{E2.1} with $b=a$. Note also that $\lim_{t\to 0^+}\|v(\cdot,t)- u^2_0\|_{L_2(B(x_0,1))}\to 0$ for all $x_0\in \R^3$, $u^2_0|_{B(4/3) = 0}$ (thus we have as a byproduct $\lim_{t\to 0^+}\|v(\cdot,t)\|_{L_2(B(4/3))} = 0$) and from Lemma \ref{L3.1}, there exists $0<T_2 = T_2(\al,\kappa,m,M)<T_1$ such that
\begin{equation}\label{E3.3}
    \essup_{0<t<T_2}\frac{1}{2}\int_{B(2)}|v(x,t)|^2 dx + \int_0^{T_2}\int_{B(2)}|\nabla v(x,t)|^2 dx dt \leq C(\al,m,M).
\end{equation}
From the local energy \eqref{E3.2} for $v$ and the fact that $\lim_{t\to 0^+}\|v(\cdot,t)\|_{L_2(B(4/3))} = 0$, we obtain
\begin{multline}
    \frac{1}{2}\int_{B(4/3)}|v(x,t)|^2 \phi(x) dx + \int_0^t \int_{B(4/3)}|\nabla v|^2 \phi dz + \int_0^t\int_{B(4/3)}\kappa(\dvg v)^2\phi dz\\ \leq \int_0^t \int_{B(4/3)}\frac{|v|^2}{2}\DD \phi dz + \int_0^t \int_{B(4/3)}\frac{|v|^2}{2}(v + a)\cdot \nabla \phi dz\\ + \int_0^t \int_{B(4/3)}a\otimes v:\nabla v \phi dz + \int_0^t \int_{B(4/3)} a\otimes v:v\otimes \nabla \phi dz\\ - \kappa\int_0^t \int_{B(4/3)}v\cdot\nabla \phi \dvg v dz + \frac{1}{2}\int_0^t \int_{B(4/3)}\phi v\cdot a \dvg v dz,     
\end{multline}
for a.e. $t\in (0,T_2)$ and $0\leq \phi \in C^{\infty}_0(B(4/3))$ such $\phi \equiv 1$ in $B$. The previous estimate and a repetitive use of H\"older inequality (with estimate \eqref{E3.3} at hand) yield:
\begin{multline}\label{E3.5}
    \int_B\frac{|v(x,t)|^2}{2}dx + \int_0^t \int_B |\nabla v|^2 dz + \int_0^t \int_B \kappa (\dvg v)^2 dz\\
    \leq C(\al,\kappa,m,M)(t + t^{\frac{1}{10}} + t^{\frac{2m-3}{5 m}} + t^{\frac{m-3}{5m}}),
\end{multline}
for a.e. $t\in (0,T_2)$.\\
Now, fix $t_0\in (0,T_2)$ to be specified later. Then extend $v$ to $B\times (-1 + t_0,t_0)$ by setting $v \equiv 0$ in $B\times (-1 + t_0,0)$. Extend also $a$ to $B\times (-1 + t_0,t_0)$ by setting $a\equiv 0$ for $t<0$. Clearly the extended function $v$ is a suitable weak solution to \eqref{E2.1}, with the extended $a$, in $B\times (-1 + t_0,t_0)$. Indeed, the fact that $\lim_{t\to 0^+}\|v(\cdot,t)\|_{L_2(B)} = 0$ insure that $\partial_t v$ and $\partial_t \frac{|v|^2}{2}$ will not cause any problem across $\{t=0\}$. Finally, because of \eqref{E3.5}, if we choose $t_0 = t_0(\al,\kappa,m,M)<T_2$ sufficiently small, we can apply Theorem \ref{Thm2.7} and conclude that $v$ is H\"older continuous in $B(1/2)\times [0,t_0]$, for some $\gamma = \gamma(m)\in (0,1)$. This concludes the proof of the theorem.  
\end{proof}
Theorem \ref{Thm3.2} allows us to prove the following.
\begin{theorem}[Local H\"older regularity of Leray solutions]\label{Thm3.3}
Let $u_0 \in L_{2,loc}(\R^3)$ such that $\al:= \sup_{x_0 \in \R^3}\int_B |u_0(x)|^2 dx < \infty$. Suppose in addition that $M:= \|u_0\|_{C^{\gamma,\frac{\gamma}{2}}(B(2))}<\infty$. Then, there exists $T=T(\al,\gamma,\kappa,M)>0$ such that any Leray weak solution $u$ to \eqref{Toy-Mod1} satisfies:
\[
\|u\|_{C^{\gamma,\frac{\gamma}{2}}(\overline{B(1/4)}\times [0,T])} \leq C(\al,\gamma,\kappa,M).
\]
\end{theorem}

\begin{proof}[Sketch of proof]
With the same notation as in Theorem \ref{Thm3.2}, we have that supp$u^1_0 \subset\subset B(2)$ and $\|u_0^1\|_{C^{\gamma}(\R^3)} \leq CM$. Consequently, $u-a$ is H\"older continuous with some exponent $\beta \in (0,\gamma)$ in $\overline{B(1/2)}\times [0,T_1]$ where $T_1 = T_1(\al,\kappa,M)>0$. Since the initial data $u_0^1$ for $a$ is in $C^{\gamma}(\R^3)$, it is not difficult to show that $a \in C^{\gamma,\frac{\gamma}{2}}(\R^3\times [0,T_1])$. Therefore, $u$ is H\"older continuous with exponent $\beta$ in $\overline{B(1/2)}\times [0,T_1]$. From this point, a standard bootstrap argument with repetitive use of Lemma \ref{L2.2} yields the required H\"older continuity of $u$; and a careful track of the constants gives us the estimate in the theorem. This concludes the proof.
\end{proof}
\section{Proof of Theorem \ref{Thm1.1}}
We are ready now to proof the main result of this paper.
\begin{proof}[Proof of Theorem \ref{Thm1.1}]
From Lemma \ref{L3.1}, we have that 
\begin{equation}\label{E4.1}
    \sup_{0<t<T_1} \frac{1}{2}\int_B |u(x,t)|^2 dx + \int_0^{T_1}\int_B |\nabla u(x,t)|^2 dx dt \leq C(\kappa, \|u_0\|_{C(\partial B)}),
\end{equation}
with $T_1= T_1(\kappa,\|u_0\|_{C(\partial B)})$.\\
Since $u$ is scale invariant i.e. $\lbd u(\lbd x, \lbd^2 t) = u(x,t)$, we have that
\[
u(x,t) = \frac{1}{\sqrt{t}}U\left( \frac{x}{\sqrt{t}} \right),\quad t>0,
\]
where $U(\cdot) := u(\cdot,1)$. Thus, we deduce from \eqref{E4.1} that
\begin{equation}\label{E4.2}
    \sqrt{t^*}\int_{B(1/\sqrt{t^*})} |U(y)|^2 dy + \sqrt{t^*}\int_{B(1/\sqrt{t^*})} |\nabla U(y)|^2 dy \leq C(\kappa, \|u_0\|_{C(\partial B)}),
\end{equation}
for all $t^*\in (0,T_1)$.\\
On the other hand, for all $|x_0| = 8$, we have $u_0 \in C^{\infty}(B(x_0,4))$. Therefore, by Theorem \ref{Thm3.3} (and some simple bootstrapping arguments), we have that there exists $T_2 = T_2(\kappa,u_0)>0$ such that 
\begin{equation}\label{E4.3}
    \|\partial_t \partial^{\al} u\|_{L_{\infty}(\overline{B(x_0,1/8)}\times [0,T_2])} \leq C(\al,\kappa,u_0),
\end{equation}
for all Leray solution $u$ to system \eqref{Toy-Mod1} with $u_0$ initial data.\\
Since for all $\lbd>0$, $u^{\lbd}:(x,t)\mapsto \lbd u(\lbd x,\lbd^2 t)$ is also a Leray solution to \eqref{Toy-Mod1} with initial data $u_0$, then \eqref{E4.3} holds also for $u^{\lbd}$ and we obtain that
\[
|\lbd^{1+|\al|}\partial^{\al} u(\lbd x_0,\lbd^2 t) - \partial^{\al}u_0(x_0)| \leq C(\al,\kappa,u_0)t.
\]
Setting $y=x_0/\sqrt{t}$, and by using the homogeneity of $\partial^{\al} u_0$, we get that:
\begin{equation}
    |\partial^{\al}(U - u_0)(y)| \leq \frac{C(\al,\kappa,u_0)}{|y|^{3+|\al|}},\quad \forall |y|> \frac{8}{\sqrt{T_2}}.
\end{equation}
Now, we choose $t^* = t^*(\kappa,u_0)$ in \eqref{E4.2} sufficiently small so that
\begin{equation}\label{E4.5}
    \int_{B(\frac{16}{\sqrt{T_2}})}\left(|U(y)|^2 + |\nabla U(y)|^2 dy \right) \leq C(\kappa,u_0);
\end{equation}
and because $u$ satisfies \eqref{Toy-Mod1}, it's not difficult to see that
\begin{equation}\label{E4.6}
    -\DD U - \kappa \nabla \dvg U + U\cdot\nabla U + \frac{U}{2}\dvg U - \frac{x}{2}\cdot \nabla U - \frac{U}{2} = 0\quad\mbox{in }\R^3.
\end{equation}
Thus, from Elliptic estimates (alongside ideas we used in the proof of Lemma \ref{L2.2}), we find that
\begin{equation}\label{E4.7}
    \|U\|_{C^k(\overline{B(9/\sqrt{T_2})})} \leq C(k,u_0)\quad (k=0,1,2\ldots)
\end{equation}
Finally, let us explain how we define the semigroup $S_{\kappa}(t)$ and derive the required estimates to close the proof of the theorem.\\
The classical Calderon-Zygmund combined with real interpolation methods allow us to get the existence of a unique function (up to a constant) $q_0$ such that $\DD q_0 = \dvg u_0$ and
\begin{equation}\label{E4.8}
    \|q_0\|_{BMO(\R^3)} + \|\nabla q_0\|_{L_{3,\infty}(\R^3)} \leq c\|u_0\|_{C(\partial B)}.
\end{equation}
Set $u_0^{(1)} := \nabla q_0$ and notice that, because of the uniqueness of $q_0$ and the scaling symmetry of $u_0$, we have that $u_0^{(1)}$ is also $(-1)$-homogeneous. Next, by elliptic estimates we see that $u_0^{(1)}\in C^{\infty}(\partial B)$ and we have (thanks to \eqref{E4.8})
\[
|\partial^{\al} u_0^{(1)}(x)| \leq \frac{C(\al,u_0)}{|x|^{1+|\al|}}.
\]
Introduce now $u_0^{(0)} := u_0 - u^{(1)}_0$ and notice that $\dvg u^{(0)}_0 = 0$ (we should also point out that $\curl u_0^{(1)} = 0$ by definition). We set
\begin{equation}
    S_{\kappa}(t)u_0 = e^{\DD t}u_0^{(0)} + e^{(1+\kappa)\DD t}u_0^{(1)};
\end{equation}
It is clear from this definition that $S_{\kappa}(t)$ is a semigroup. Moreover, using similar arguments as the ones we use in the proof of Lemma \ref{L2.2}, we see that $v(x,t) := S_{\kappa}(t)u_0(x)$ solves the Lam\'e system
\begin{equation}
    \left\{
    \begin{gathered}
    \partial_t v - \DD v - \kappa \nabla\dvg v = 0\quad \mbox{in }\R^3\times (0,\infty)\\
    v|_{t=0} = u_0\quad \mbox{in }\R^3.
    \end{gathered}
    \right.
\end{equation}
Also, we have $\lbd v(\lbd x,\lbd^2 t) = v(x,t)$ for all $\lbd >0$. So, if we set $V := S_{\kappa}(1)u_0$, we have that
\[
v(x,t) = \frac{1}{\sqrt{t}}V\left( \frac{x}{\sqrt{t}}\right)\quad \forall (x,t)\in \R^3\times (0,\infty).
\]
and finally
\[ 
|\partial^{\al}(S_{\kappa}(1)u_0 - u_0)| \leq |\partial^{\al}(e^{\DD}u_0^{(0)} - u_0^{(0)})| + |\partial^{\al}(e^{(1+\kappa)\DD}u_0^{(1)} - u_0^{(1)})|\leq \frac{C(\al,\kappa,u_0)}{(1+|x|)^{3+|\al|}}
\]
by well-known properties of the heat equation. Same machinery for system \eqref{Toy-Mod2}. And this concludes the proof.
\end{proof}
\paragraph{Acknowledgement}
This work was supported by the Engineering and Physical Sciences Research Council [EP/L015811/1]. The author would like to thank Gregory Seregin for the insightful discussions during the completion of this paper.           
\newpage
\bibliographystyle{siam}
\bibliography{refs}

\begin{thebibliography}{10}

\bibitem{Evans86}
{\sc L.~Evans}, {\em {Quasiconvexity and partial regularity in the calculus of
  variations}}, Arch. Ration. Mech. Anal., 95 ({1986}), pp.~227--252.

\bibitem{Jia13}
{\sc {H. Jia, V. \v{S}ver\'{a}k}}, {\em {Minimal $L_3$-initial data for
  Potential Navier-Stokes Singularities}}, SIAM J. MATH. ANAL., 45 ({2013}),
  pp.~1448--1459.

\bibitem{Jia14}
\leavevmode\vrule height 2pt depth -1.6pt width 23pt, {\em {Local-in-space
  estimates near initial time for weak solutions of the Navier-Stokes equations
  and forward self-similar solutions}}, Invent math, 196 ({2014}),
  pp.~233--265.

\bibitem{Hou19}
{\sc F.~Hounkpe}, {\em {On a Toy-Model related to the Navier-Stokes
  Equations}}, Journal of Mathematical Sciences,  ({To appear}).

\bibitem{Kry01}
{\sc N.~Krylov}, {\em {The Heat equation in $L_q((0,T),L_p)$-spaces with
  weights}}, Siam J. Math. Anal., 32 ({2001}), pp.~1117--1141.

\bibitem{Kry08}
\leavevmode\vrule height 2pt depth -1.6pt width 23pt, {\em {Lectures on
  Elliptic and Parabolic Equations in Sobolev Spaces}}, vol.~96, American
  Mathematical Society, {Graduate Studies in Mathematics}~ed., 2008.

\bibitem{Caff82}
{\sc {L. Caffarelli, R.V. Kohn, L. Nirenberg}}, {\em {Partial regularity of
  suitable weak solutions of the Navier-Stokes equations}}, Comm. Pure Appl.
  Math., XXXV (1982), pp.~771--831.

\bibitem{Esc03}
{\sc {L. Escauriaza, G. Seregin, V. \v{S}ver\'{a}k}}, {\em
  {$L_{3,\infty}$-solutions of the Navier-Stokes Equations and Backward
  Uniqueness}}, Russian Mathematical Surveys, Vol. 58 (2003), pp.~211--250.

\bibitem{Lemar02}
{\sc Lemari\'e-Rieusset}, {\em {Recent Developments in the Navier-Stokes
  Problem}}, vol.~431, Chapman \& Hall, Boca Raton, research notes in
  mathematics~ed., 2002.

\bibitem{Lin98}
{\sc F.-H. Lin}, {\em {A New Proof of the Caffarelli-Kohn-Nirenberg Theorem}},
  Comm. Pure Appl. Math., 51 (1998), pp.~241--257.

\bibitem{Lady68}
{\sc {O. A. Ladyzhenskaya, V. A. Solonnikov, N. N. Uraltseva}}, {\em {Linear
  and quasi-linear equations of parabolic type}}, vol.~23, American
  Mathematical Society, Providence, R.I., {Translations of mathematical
  monographs}~ed., 1968.

\bibitem{Rus12}
{\sc W.~Rusin}, {\em {Incompressible 3D Navier-Stokes Equations as a Limit of a
  Nonlinear Parabolic System}}, Journal of Mathematical Fluid Mechanics, 14
  (2012), pp.~383--405.

\bibitem{Ser14}
{\sc G.~Seregin}, {\em {Lecture Notes on Regularity Theory for the
  Navier-Stokes Equations}}, World Scientific, 2014.

\end{thebibliography}
\end{document}